\documentclass{scrartcl}
\usepackage[T1]{fontenc}
\usepackage[english]{babel}
\usepackage[latin1]{inputenc}
\usepackage{amsmath}
\usepackage{amsthm,amssymb}
\usepackage{dsfont}
\usepackage{csquotes}
\usepackage{graphicx}
\usepackage{color, pgf}
\usepackage{hyperref}
\usepackage{url}
\hypersetup{colorlinks=true, urlcolor=blue, citecolor=red} 
\newtheorem{thm}{Theorem}[section]
\newtheorem{defn}[thm]{Definition}
\newtheorem{lem}[thm]{Lemma}

\newtheorem{rem}[thm]{Remark}

\newcommand{\capa}{\operatorname{cap}}
\newcommand{\R}{\mathbb{R}}

\newcommand{\W}{\mathcal{W}}

\begin{document}

\title{Uniqueness results of a nonlinear stochastic diffusion-convection equation with reflection}
\author{Niklas Sapountzoglou \thanks{Institute of Mathematics, Clausthal University of Technology, 38678 Clausthal-Zellerfeld, Germany \href{mailto:niklas.sapountzoglou@tu-clausthal.de}{\texttt{niklas.sapountzoglou@tu-clausthal.de}}}}
\date{}

\maketitle

\begin{abstract}
We are interested in the uniqueness of solutions of a nonlinear, pseudomonotone, stochastic diffusion evolution problem with homogeneous Dirichlet boundary conditions with reflection, where the noise term is additive and given by a stochastic It\^{o} integral with respect to a Hilbert space valued cylindrical Wiener process. In fact, since there is no It\^{o} formula available for a solution in general, a general uniqueness result seems not to be available. Nevertheless, assuming more regularity for the solutions or the reflection, we may show some comparison principles.
\end{abstract}
\noindent
\textbf{Keywords:} Stochastic PDEs $\bullet$ Obstacle problem 
$\bullet$ Variational inequalities $\bullet$ \\
additive noise \\

\noindent
\textbf{Mathematics Subject Classification (2020):} 35R60 $\bullet$ 60H15 $\bullet$ 35K55 $\bullet$ 35R35

\section{The first section of your paper}\label{sec1}
\subsection{Introduction}\label{sec 1.1}
We want to study uniqueness or, more precisely, comparison principles of solutions of the nonlinear stochastic obstacle problem with additive colored noise
\begin{align}\label{Obstacle problem1}
		\begin{cases}
			du \ -\operatorname{div}\, a(\cdot,u,\nabla u)\,dt + f(u) \, dt= \Phi \,dW +d\eta &\text{ in } Q_T\\
			u(0,\cdot)=u_0\geq 0 &\text{ in } D\\
			u=0 &\text{ on } (0,T) \times \partial D\\
			u\geq 0, ~\eta \geq 0 \text{ and } \eta(u)=0,	 
		\end{cases} \tag{P}
	\end{align}
where the spatial domain $D$ is a bounded Lipschitz domain in $\mathbb{R}^d$, $T>0$ and $Q_T:= (0,T) \times D$. In this sequel, $\eta$ plays the role of a reflection measure, i.e., $\eta$ prevents $u$ from becoming negative. The condition $\eta(u)=0$ secures that the reflection is minimal, i.e., $\eta$ is equal to $0$ whenever $u>0$. The exact meaning of the expression $\eta(u)$ is given below. Problems of type \eqref{Obstacle problem1} have been studied in the past decades in the deterministic case (see e.g. \cite{DuvautLions72}) and for obstacles $\psi$ of Lewy-Stampacchia type in \cite{GuibeMTV2020,GUIBE2024,MTV19}. Stochastic obstacle problems with obstacles $\psi$ of Lewy-Stampacchia type have been studied in \cite{STVZ23, TahraouiVallet2022}. Therein, existence and uniqueness of a pair $(u, \eta)$ solving \eqref{Obstacle problem1} with obstacle $\psi$ instead of $0$ has been shown for different types of nonlinear operators and additive forcing $f$. The construction of the obstacle leads to the fact that the reflection measure $\eta$ is a functional in $L^{p'}(0,T; W^{-1,p'}(D))$, where $p > 1$ depends on the nonlinear operator and $p' = \frac{p}{p-1}$. In this case, an It\^{o} formula for the equation in \eqref{Obstacle problem1} is available, but this theory does not cover the case $\psi \equiv 0$. The case $\psi \equiv 0$ has been studied for $d=1$ for linear operators $- \operatorname{div} a(u, \nabla u)$ (see, e.g. \cite{Pardoux1993,Haussmann1989,Pardoux1992}). In those works, existence and uniqueness of solutions have been studied, where the uniqueness results in \cite{Pardoux1993,Pardoux1992} heavily rely on the fact that a linear problem is considered, the uniqueness result in \cite{Haussmann1989} a-priori only holds for regular solutions. The uniqueness result therein has been generalized by establishing an It\^{o} formula in \cite{FP}. We mention that due to the fact that only one spatial dimension is considered, the solution $u$ is continuous with respect to space-time. In higher spatial dimensions the solution fails to be continuous in general. As a consequence, an It\^{o} formula as stated in \cite{FP} fails to hold in general which makes it difficult to obtain a uniqueness result in the general case (see Section \ref{Difficulties}). Moreover, the arbitrary spatial dimension makes it challenging to find a proper meaning to the expression $\eta(u)$ since $\eta$ is a non-negative random Radon measure with respect to $Q_T$ and therefore $\eta$ is defined on Borel measurable functions only. Hence, a proper Borel measurable representative of a solution $u$ has to be determined which is non-negative $\eta$-a.e. in $Q_T$ and satisfies $\eta(u) := \int_{Q_T} u \, d\eta=0$. To this end, the theory of parabolic capacity (see e.g. \cite{DPP,Pet08,PPP11}) has to be used. The existence of a pair $(u, \eta)$ solving the obstacle problem \eqref{Obstacle problem1} has been established in \cite{STVZ24} under high regularity assumptions on the additive noise term.

\subsection{Notations and assumptions}
For $d\in\mathbb{N}$ with $d \geq 2$ let $D\subset \mathbb{R}^d$ be a bounded domain with Lipschitz boundary, $T>0$ and $Q_T:=(0,T)\times D$. Moreover we assume $p \geq 2$ and $p' = \frac{p}{p-1}$. The closure of $\mathcal{C}_c^{\infty}(D)$ in $W^{1,p}(D)$ will be denoted by $W_0^{1,p}(D)$, where $W^{1,p}(D)$ is the standard Sobolev space. The topological dual of $W_0^{1,p}(D)$ will be denoted by $W^{-1,p'}(D)$. Moreover, let $(\Omega,\mathcal{F},\mathds{P}, (\mathcal{F}_t)_{t\geq 0}, (W_t)_{t \geq 0})$ be a stochastic basis with a complete, right continuous filtration and a $L^2(D)$-valued cylindrical Wiener process $(W_t)_{t\geq 0}$.
The problem of our interest may formally be written as: find a pair $(u,\eta)$ which is a solution to
	\begin{align}\label{Obstacle problem}
		\begin{cases}
			du \ -\operatorname{div}\, a(\cdot,u,\nabla u)\,dt + f(u) \, dt= \Phi \,dW +d\eta &\text{ in } Q_T\\
			u(0,\cdot)=u_0\geq 0 &\text{ in } D\\
			u=0 &\text{ on } (0,T) \times \partial D\\
			u\geq 0, ~\eta \geq 0 \text{ and } \int_{Q_T} u \, d\eta=0,	 
		\end{cases}
	\end{align}
where $f: \R \to \R$ is Lipschitz continuous with Lipschitz constant $L_f\geq 0$ such that $f(0)=0$ and $a:D\times\mathbb{R}^{d+1} \to \mathbb{R}^d$ is a Carath\'eodory function satisfying the following assumptions:
	\begin{itemize}
		\item[$(A1)$] Monotonicity: \[(a(x,\lambda,\xi)-a(x,\lambda,\zeta))\cdot(\xi-\zeta)\geq 0\]
		for all $\lambda\in\mathbb{R}$, $\xi,\zeta\in\mathbb{R}^d$ and almost every $x\in D$.
		\item[$(A2)$] Coercivity and growth: There exist $\kappa\in L^1(D)$, $C_1>0$, $C_2\geq 0$, $C_3\geq 0$ and a nonnegative function $g\in L^{p'}(D)$ such that 
		\[a(x,\lambda,\xi)\cdot \xi\geq \kappa(x)+C_1|\xi|^{p},\]
		\[|a(x,\lambda,\xi)|\leq C_2|\xi|^{p-1}+C_3|\lambda|^{p-1}+g(x)\]
		for all $\lambda\in\mathbb{R}$, $\xi\in\mathbb{R}^d$ and almost every $x\in D$.
		\item[$(A3)$] Lipschitz regularity: There exists $C_4\geq 0$ and a nonnegative function $h\in L^{p'}(D)$ such that 
		\[|a(x,\lambda_1,\xi)-a(x,\lambda_2,\xi)|\leq \Big(C_4|\xi|^{p-1}+h(x)\Big)|\lambda_1-\lambda_2|\]
		for all $\lambda_1$, $\lambda_2\in \mathbb{R}$, for all $\xi\in\mathbb{R}^d$ and almost all $x\in D$.
	\end{itemize}
Then, the operator $W_0^{1,p}(D) \ni u \mapsto - \operatorname{div} a(u, \nabla u) \in W^{-1,p'}(D)$ is a pseudomonotone Leray-Lions operator. Moreover, we assume $\Phi: \Omega \times (0,T) \to HS(L^2(D))$ to be progressively measurable and $\Phi \in L^2(\Omega \times (0,T); HS(L^2(D)))$. In \eqref{Obstacle problem} the reflection $\eta$ is an unknown of the problem and a non-negative random Radon-measure in general.

\subsection{Notion of a solution}\label{sec 1.3}
\begin{defn}
A random Radon measure $\eta$ on $Q_T$ is called weak-$\ast$ adapted if and only if, for any $\psi\in \mathcal{C}_c(Q_T)$ and  any $t\in(0,T)$, the mapping 
	\begin{align*}
		\Omega \ni \omega \mapsto \int_{(0,t]\times D} \psi(s,x)  \ d\eta_{\omega}(s,x) \quad \text{is } \mathcal{F}_t \text{-measurable}.
	\end{align*}	
\end{defn}

\begin{defn}
	For $2\leq p<\infty$ we define the space
		\[\W:=\{v\in L^p(0,T;W^{1,p}_0(D)) \ | \ \partial_t v \in L^{p'}(0,T;W^{-1,p'}(D))\}.\]
	It is a separable reflexive Banach space endowed with its natural norm
		\begin{equation*}
			\Vert v \Vert_{\W}^p := \Vert v \Vert_{L^p(0,T; W_0^{1,p}(D))}^p + \Vert \partial_t v \Vert_{L^{p'}(0,T; W^{-1,p'}(D))}^p.
		\end{equation*}
	Moreover, $\W$ is continuously imbedded in $\mathcal{C}([0,T]; L^2(D))$ and any $v \in \W$ admits a $\capa_p$-quasi continuous representative (see Section \ref{Appendix}, Appendix for more details about parabolic capacity and $\capa_p$-quasi continuous functions).
	Its dual space will be denoted by $\W'$.
\end{defn}

\begin{defn}\label{Solution}
Let $u_0 \in L^2(\Omega \times D)$ be $\mathcal{F}_0$-measurable and $u_0 \geq 0$ a.e. in $\Omega \times D$. A pair $(u, \eta)$ is a solution to \eqref{Obstacle problem}, if and only if
\begin{itemize}
	\item[$i.)$] $u \geq 0$ a.e. in $\Omega \times Q_T$ and
	\begin{align*}
		u \in L_w^2(\Omega,L^\infty(0,T,L^2(D)))\cap L^p(\Omega; L^p(0,T; W_0^{1,p}(D))),
	\end{align*}
	\item[$ii.)$] there exists a right-continuous in time representative of $u$ with values in $L^{2}(D)$ satisfying $u(t=0^+)=u_0$, 
	\item[$iii.)$] there exists $\eta \in L^{p^\prime}(\Omega, \W^\prime)$ that is also  a weak-$\ast$ adapted random non-negative Radon measure on $Q_T$ with, $\mathds{P}$-a.s. 
\begin{align*}
	\partial_t\Big[u-\int_0^{\cdot} \Phi \, dW\Big] - \operatorname{div} a(\cdot, u , \nabla u) + f(u) = \eta \quad \text{in }\W^\prime,
\end{align*}
\item[$iv.)$] $\mathds{P}$-a.s. in $\Omega$ there exists a quasi everywhere defined representative of $u$ in $Q_T$, non-negative quasi everywhere in $Q_T$,  such that $\int_{Q_T} u \ d\eta =0$.
\end{itemize}
 \end{defn}	

\section{Comparison principles for regular solutions}\label{sec2}

\begin{thm}\label{Contraction L2}
  Let the function $a$ be independent of the second variable and $(u, \eta)$ and $(v, \nu)$ be two solutions of \eqref{Obstacle problem} with initial values $u_0 \in L^2(\Omega \times D)$ and $v_0 \in L^2(\Omega \times D)$ $\mathcal{F}_0$-measurable and non-negative a.e. in $\Omega \times D$ respectively. Moreover, let $\mathbb{P}$-a.s. $\eta, \nu \in L^{p'}(0,T;W^{-1,p'}(D))$. Then there exists a constant $C>0$ only depending on $T$ and $f$ such that $\mathbb{P}$-a.s.
  \begin{align*}
  	 \sup\limits_{t \in [0,T]} \Vert u(t) - v(t) \Vert_{L^2(D)} \leq C \Vert u_0 - v_0 \Vert_{L^2(D)}.
  \end{align*}
\end{thm}
\begin{rem}
	We remark that in the case in which $(u, \eta)$ is a solution to \eqref{Obstacle problem} and the reflection $\eta$ is an element of $L^{p'}(0,T;W^{-1,p'}(D))$, condition $iv.)$ in Definition \ref{Solution} reads as
	\begin{align*}
		\langle \eta , u \rangle_{L^{p'}(0,T;W^{-1,p'}(D)), L^p(0,T;W_0^{1,p}(D))} = 0.
	\end{align*}
	Moreover, for any $v \in \W$ we have
	\begin{align*}
		\langle \eta , v \rangle_{\W', \W} = \langle \eta , v \rangle_{L^{p'}(0,T;W^{-1,p'}(D)), L^p(0,T;W_0^{1,p}(D))}.
	\end{align*}
\end{rem}
\begin{proof}
  For $t \in (0,T]$ we set $X_t:= L^p(0,t;W_0^{1,p}(D))$ and denote by $X_t'$ its topological dual. The expression $u-v$ satisfies $\mathbb{P}$-a.s. the following equality:
  	\begin{align}\label{eq 051224_01}
  		\partial_t(u-v) - \operatorname{div} (a(\nabla u) - a(\nabla v)) + f(u) - f(v) = \eta - \nu
  	\end{align}
	in $X_T$. Hence, $\partial_t (u-v) \in L^{p'}(0,T; W^{-1,p'}(D))$ and therefore $u-v \in \W$. We set $Q_t := (0,t) \times D$ and
	\begin{equation*}
		\W(0,t):= \{v\in L^p(0,t;W^{1,p}_0(D)) \ | \ \partial_t v \in L^{p'}(0,t;W^{-1,p'}(D))\}.
	\end{equation*}
	Testing \eqref{eq 051224_01} with $u-v$ yields $\mathbb{P}$-a.s. for all $t \in [0,T]$:
	\begin{align*}
		&\frac{1}{2} \int_D |u(t) - v(t)|^2 - |u_0- v_0|^2 \ dx + \int_0^t \int_D (a(\nabla u) - a(\nabla v)) \nabla(u-v) \ dx \ ds\\
		&+ \int_0^t \int_D (f(u) - f(v)) (u-v) \ dx \ ds = \langle \eta - \nu, u - v\rangle_{\W(0,t)', \W(0,t)} = \langle \eta - \nu, u - v\rangle_{X_t', X_t}\\
		&\Leftrightarrow I + II + III = IV.
	\end{align*}
	By Assumption (A1) we have $II \geq 0$. Moreover, we have
	\begin{align*}
		 III  \geq -L_f \int_0^t \int_D | u-v|^2 \ dx \ ds 
	\end{align*}
	and
	\begin{align*}
		 IV =  \langle \eta , u \rangle_{X_t', X_t} - \langle \eta , v\rangle_{X_t', X_t} - \langle \nu, u \rangle_{X_t', X_t} + \langle \nu, v\rangle_{X_t', X_t} = - \langle \eta , v\rangle_{X_t', X_t} - \langle \nu, u \rangle_{X_t)', X_t} \leq 0
	\end{align*}
	since $\eta, \nu, u,$ and $v$ are non-negative. Hence, $\mathds{P}$-a.s. for all $t \in [0,T]$ we get
	\begin{align*}
		\int_D |u(t) - v(t)|^2 \ dx \leq \int_D |u_0 - v_0|^2 \ dx + 2L_f \int_0^t \int_D | u-v|^2 \ dx \ ds.
	\end{align*}
	Now, Gronwall's lemma yields $\mathbb{P}$-a.s. for all $t \in [0,T]$:
	\begin{align*}
		\int_D |u(t) - v(t)|^2 \ dx \leq \exp(2TL_f) \int_D |u_0 - v_0|^2 \ dx.
	\end{align*}
\end{proof}
Now we establish a $L^1$-contraction principle in the case where the function $a$ is allowed to depend on the second variable.
\begin{thm}\label{L1 contraction}
  Let (A1) - (A3) be satisfied and $(u, \eta)$ and $(v, \nu)$ be two solutions of \eqref{Obstacle problem} with initial values $u_0 \in L^2(\Omega \times D)$ and $v_0 \in L^2(\Omega \times D)$ $\mathcal{F}_0$-measurable and non-negative a.e. in $\Omega \times D$ respectively. Moreover, let $\mathbb{P}$-a.s. $\eta, \nu \in L^{p'}(0,T;W^{-1,p'}(D))$. Then there exists a constant $C>0$ only depending on $T$ and $f$ such that $\mathbb{P}$-a.s.
  \begin{align*}
  	 \sup\limits_{t \in [0,T]} \Vert u(t) - v(t) \Vert_{L^1(D)} \leq C \Vert u_0 - v_0 \Vert_{L^1(D)}.
  \end{align*}
\end{thm}
\begin{proof}
  This proof is inspired by \cite{STVZ23}. For $\delta >0$ let us consider $n_{\delta} \in C^{\infty}(\R)$ such that
  \begin{itemize}
  	\item $n_{\delta}(r)=r -2\delta$ for $r> \delta$,
  	\item $n_{\delta}(r)= n_{\delta}(-r)$ for all $r \in \R$,
	\item $0 \leq n_{\delta}'(r) \leq 1$ for $r \geq 0$,
	\item $0 \leq n_{\delta}'' \leq \frac{C_5}{\delta}$ for some constant $C_5>0$. 
  \end{itemize}
Then, $\operatorname{supp} (n_{\delta}'') \subset [-\delta , \delta]$. By classical chain rule in Sobolev spaces we have $n_{\delta}'(u-v) \in L^p(0,T;W_0^{1,p}(D))$. Testing 
  \begin{align}\label{eq 121224_01}
  	\partial_t(u-v) - \operatorname{div} (a(u, \nabla u) - a(v, \nabla v)) + f(u) - f(v) = \eta - \nu
  \end{align}
  with $n_{\delta}'(u-v)$ yields $\mathbb{P}$-a.s.
  \begin{align}\label{eq 091224_01}
  I + II + III = IV,
  \end{align}
  where
  \begin{align*}
  	I&= \frac{1}{2} \int_D n_{\delta}(u(t) - v(t)) - n_{\delta}(u_0 - v_0) \ dx \to \frac{1}{2} \int_D |u(t) - v(t)| - |u_0 - v_0| \ dx \text{ as } \delta \to 0^+, \\
  	II&= \int_0^t \int_D (a(u, \nabla u) - a(v, \nabla v)) \nabla (n_{\delta}'(u-v)) \ dx \ ds \\
  	&=  \int_0^t \int_D (a(u, \nabla u) - a(v, \nabla v)) n_{\delta}''(u-v) \nabla (u-v) \ dx \ ds \\
  	&=  \int_0^t \int_D (a(u, \nabla u) - a(u, \nabla v)) n_{\delta}''(u-v) \nabla (u-v) \ dx \ ds \\
  	&+ \int_0^t \int_D (a(u, \nabla v) - a(v, \nabla v)) n_{\delta}''(u-v) \nabla (u-v) \ dx \ ds \\
  	&\geq \int_0^t \int_D (a(u, \nabla v) - a(v, \nabla v)) n_{\delta}''(u-v) \nabla (u-v) \ dx \ ds \\
  	&\geq - \int_0^t \int_D (C_4 |\nabla v|^{p-1} + h)|u-v| n_{\delta}''(u-v) |\nabla (u-v)| \ dx \ ds \\
  	&\geq - \int_{\{|u-v| \leq \delta\}} (C_4 |\nabla v|^{p-1} + h) |u-v| \frac{C_5}{\delta} |\nabla (u-v)| \ dx \ ds \\
  	&\geq - \int_{\{|u-v| \leq \delta\}} (C_4 |\nabla v|^{p-1} + h) C_5 |\nabla (u-v)| \ dx \ ds \to 0 \text{ as } \delta \to 0^+
  \end{align*}
  and
  \begin{align*}
  	III&= \int_0^t \int_D (f(u) - f(v)) n_{\delta}'(u-v) \ dx \ ds \geq - L_f \int_0^t \int_D |u-v| \ dx \ ds.
  \end{align*}
  Since $n_{\delta}(-r) = n_{\delta}(r)$, we have $n_{\delta}'(-r)= - n_{\delta}'(r)$. Moreover, we have
  \begin{align*}
  	IV= \langle \eta - \nu, n_{\delta}'(u-v) \rangle_{X_t', X_t} = \langle \eta, n_{\delta}'(u-v) \rangle_{X_t', X_t} + \langle \nu, n_{\delta}'(v-u) \rangle_{X_t', X_t}.
  \end{align*}
  In order to show $IV \leq 0$ for any $\delta >0$, because of symmetry reasons we only need to show $ \langle \eta, n_{\delta}'(u-v) \rangle_{X_t', X_t} \leq 0$. Set $w_{\delta} := u - \frac{\delta}{C_5} n_{\delta}'(u-v)^+$. Then $w_{\delta} \geq 0$ since for $u \leq v$ we have $w_{\delta} = u \geq 0$ and for $u \geq v$ we have
  \begin{equation*}
  	w_{\delta} = u- \frac{\delta}{C_5} n_{\delta}'(u-v) = u-v - \frac{\delta}{C_5} n_{\delta}'(u-v) + v \geq v \geq 0.
  \end{equation*}
  Now, denoting $X_t$ as in the previous proof, we obtain
  \begin{align*}
  	&\langle \eta, n_{\delta}'(u-v) \rangle_{X_t', X_t} = \langle \eta, n_{\delta}'(u-v)^+ \rangle_{X_t', X_t} - \langle \eta, n_{\delta}'(u-v)^- \rangle_{X_t', X_t} \\
  	&\leq \langle \eta, n_{\delta}'(u-v)^+ \rangle_{X_t', X_t} = \frac{C_5}{\delta} \langle \eta, u-w_{\delta} \rangle_{X_t', X_t} \\
  	&= -\frac{C_5}{\delta} \langle \eta, w_{\delta} \rangle_{X_t', X_t} \leq 0.
  \end{align*}
  Equality \eqref{eq 091224_01} yields $\mathds{P}$-a.s. for all $t \in [0,T]$
  \begin{align*}
  	\int_D |u(t) - v(t)| \ dx \leq \int_D |u_0 - v_0| + 2L_f \int_0^t \int_D |u - v| \ dx \ ds.
  \end{align*}
  Similarly as in the previous proof, Gronwall's lemma yields the result.
\end{proof}
In the following we want to get rid of the assumption that the difference of two reflections is a functional in $L^{p'}(0,T; W^{-1,p'}(D))$. To this end, we need to assume that the difference of two solutions is continuous with respect to space-time.
\begin{thm}\label{L1 contraction_new}
  Let (A1) - (A3) be satisfied and $(u, \eta)$ resp. $(v, \nu)$ be two solutions of \eqref{Obstacle problem} with initial values $u_0 \in L^2(\Omega \times D)$ and $v_0 \in L^2(\Omega \times D)$ $\mathcal{F}_0$-measurable and non-negative a.e. in $\Omega \times D$ respectively such that $u_0 - v_0 \in L^2(\Omega; \mathcal{C}_0(D))$. Moreover, let $\Omega \ni \omega \mapsto (u-v)(\omega) \in \mathcal{C}([0,T]; \mathcal{C}_0(D))$ and $\Omega \ni \omega \mapsto (\eta - \nu)(\omega) \in M(Q_T)$ be measurable. Then there exists a constant $C>0$ only depending on $T$ and $f$ such that $\mathbb{P}$-a.s.
  \begin{align*}
  	 \sup\limits_{t \in [0,T]} \Vert u(t) - v(t) \Vert_{L^1(D)} \leq C \Vert u_0 - v_0 \Vert_{L^1(D)}.
  \end{align*}
\end{thm}

\begin{proof}
According to \cite[Theorem 2.3]{FP}, an It\^{o} formula to \eqref{eq 121224_01} is available and by using $u \mapsto \int_D n_{\delta}(u) \ dx$, where $n_{\delta}$ is defined as in Theorem \ref{L1 contraction}, we obtain
	\begin{align*}
		I + II + III = IV,
	\end{align*}
where $I, II$ and $III$ coincide with the respective expressions in the previous proof and can be treated in the exact same way. Moreover, we have
	\begin{align*}
		IV= \int_{Q_t} n_{\delta}'(u-v) \, d(\eta - \nu) = \int_{Q_t} n_{\delta}'(u-v) \, d\eta + \int_{Q_t} n_{\delta}'(v-u) \, d\nu.
	\end{align*}
Similarly as in the previous proof, it is sufficient to show $\int_{Q_t} n_{\delta}'(u-v) \, d\eta\leq 0$ for $\delta >0$ and this inequality may be obtained by using the same trick as before: we set $w_{\delta}:= u - \frac{\delta}{C_5} n_{\delta}'(u-v)^+$, then $w_{\delta} \geq 0$ and
	\begin{align*}
		\int_{Q_t} n_{\delta}'(u-v) \, d\eta \leq - \frac{C_5}{\delta} \int_{Q_t} w_{\delta} \, d\eta \leq 0.
	\end{align*}
Now, Gronwall's lemma yields the result.
\end{proof}

\section{Discussion of more general uniqueness results}\label{Difficulties}
In this section we want to give an overview about the difficulties to show a uniqueness result in the general case. In general, a solution $u$ to \eqref{Obstacle problem} satisfies poor regularity properties. We cannot even expect a solution to be $\operatorname{cap}_p$-quasi continuous, especially we have $u \notin \W$ and $u - \int_0^{\cdot} \Phi \ dW \notin \W$, even for $\Phi$ satisfying arbitrary high regularity assumptions. Moreover,  the reflection $\eta$ is a locally bounded random Radon-measure on $Q_T$ only which is an element in $L^{p'}(\Omega, \W')$ in general (see \cite{STVZ24} for more information). Under these circumstances, an It\^{o} formula for $u$ is out of reach. But in the theory of well-posedness of SPDEs it is well known that an It\^{o} formula for the solution itself is crucial in order to obtain a general uniqueness result. Nevertheless, one may try to use regularization arguments to obtain an equation where an It\^{o} formula is available. For two solutions $(u, \eta)$ and $(v, \nu)$ you may consider equation \eqref{eq 121224_01} and try to test this equality by a regular approximation of $u-v$ (let us say $\Pi_n(u-v)$ for some parameter $n \in \mathbb{N}$). But, since $u-v \notin \W$ in general (actually, we only now $u-v \in L^p(0,T, W_0^{1,p}(D))$ and $\partial_t(u-v) \in \W'$), it is unclear how to pass to the limit in $\langle \partial_t (u-v), \Pi_n(u-v) \rangle_{\W', \W}$. On the other hand, you may try to regularize the equation first. The first step might be testing the whole equation \eqref{eq 121224_01} by $\varphi \in \mathcal{C}_c^{\infty}(D)$. This yields $\mathds{P}$-a.s.
\begin{align*}
	&\partial_t(u-v)\varphi - \operatorname{div} \big(\varphi(a(u, \nabla u) - a(v, \nabla v))\big) + (a(u, \nabla u) - a(v, \nabla v)) \nabla \varphi + (f(u) - f(v)) \varphi \\
	&= \eta^{\varphi} - \nu^{\varphi}
\end{align*}
in $\W'$, where $\eta^{\varphi}:= \varphi \eta$ and $\nu^{\varphi} := \varphi \nu$. Eventually, $\eta^{\varphi}$ is now a bounded Radon measure (the same applies to $\nu^{\varphi}$, respectively). Unfortunately, even this procedure seems not to help since the term $(a(u, \nabla u) - a(v, \nabla v)) \nabla \varphi$ is not Lipschitz with respect to $u$. Hence, even after more regularization steps and obtaining an equality for which an It\^{o} formula is available, an argumentation using Gronwall's lemma seems out of reach. Therefore, standard argumentations to obtain uniqueness of solutions of SPDEs seem not to work in this sequel.

\section{Appendix}\label{Appendix}
\begin{defn}[see, {\cite[Definition 2.7]{DPP}}]\label{Definition capacity}
	For any open set $U\subset Q_T$ we define the (parabolic $p$-) capacity of $U$ as
	\[\capa_p(U)=\inf\left\{\Vert v\Vert_{\mathcal{W}} \ | \ v\in\mathcal{W}, \ v\geq \mathds{1}_U \ \text{a.e. in} \ Q_T\right\}\]
	with the convention $\inf\emptyset:=+\infty$. Then, for any Borel subset $B\subset Q_T$ the definition is extended by setting
	\[\capa_p(B)=\inf\left\{\capa_p(U) \ | \ U\subset Q_T \ \text{open subset}, \  B\subset U\right\}.\]
\end{defn}
\begin{defn}[see, e.g., {\cite[p. 531]{P}}]\label{Definition cap-quasi continuous}
	A function $v:Q_T\rightarrow\mathbb{R}$ is called $\capa_p$-quasi continuous, if for every $\varepsilon>0$ there exists an open set $U_{\varepsilon}\subset Q_T$ with $					\capa_p(U_{\varepsilon})\leq \varepsilon$ and such that $v_{|Q_T\setminus U_{\varepsilon}}$ is continuous on $Q_T\setminus U_{\varepsilon}$. Especially, a $\capa_p$-quasi continuous function is Borel-measurable.
\end{defn}
\begin{lem}[see, e.g., {\cite[Lemma 2.20]{DPP}}]\label{220822_lem01}
	Any element $v$ of $\mathcal{W}$ has a $\capa_p$-quasi continuous representative $\widetilde{v}$ which is $\capa_p$-quasi everywhere unique in the sense that two $\capa_p$-quasi 	continuous representatives of $v$ are equal except on a set of zero capacity.
\end{lem}
The uniqueness part of Lemma \ref{220822_lem01} is a direct consequence of the following (stronger) result:
\begin{lem}\label{LemPositiveQuasiCont}
	If $u:Q_T\rightarrow \mathbb{R}$ is a $\capa_p$-quasi continuous function such that $u \geq 0$ a.e. in $Q_T$, then $u \geq 0$ q.e. in $Q_T$.
\end{lem}

\section*{Acknowledgements}
This work has been accepted for publication in the proceedings of the Seventeenth International Conference Zaragoza-Pau on Mathematics and its Applications, 2024.
  
\bibliographystyle{plain}
\bibliography{Sapountzoglou_uniqueness_reflection_25.bib}

\end{document}